\documentclass[11pt]{article}
\usepackage{amssymb}
\usepackage{amsmath}
\usepackage{dsfont}
\newcommand{\ball}{\mathcal{B}}
\newcommand{\C}{\mathbb{C}}
\newcommand{\car}{\mathds{1}}
\newcommand{\coleq}{\mathrel{\mathop:}=}
\newcommand{\der}{\mathrm{d}}
\newcommand{\N}{\mathbb{N}}
\newcommand{\nep}{\mathrm{e}}
\newcommand{\Rl}{\mathbb{R}}
\newcommand{\Perm}{\mathfrak{S}}
\newcommand{\Sym}{\mathrm{Sym}}
\DeclareMathOperator{\Root}{Root}
\DeclareMathOperator{\Tr}{Tr}
\newenvironment{proof}{\noindent\textbf{Proof: }}{$\square$}
\newtheorem{thm}{Theorem}

\newtheorem{rmk}[thm]{Remark}
\newtheorem{lem}[thm]{Lemma}
\newtheorem{xpl}[thm]{Example}
\newtheorem{defi}[thm]{Definition}
\newtheorem{prop}[thm]{Proposition}
\title{A theory of functions of several variables applied to square matrices}
\author{Laurent Veysseire}
\date{}
\begin{document}
\maketitle
In this paper, we give one possible definition for functions of several variables applied to endomorphisms of finite dimensional $\C$-vector spaces.
This definition is consistent with the usual notion of a function of a square matrix.
The fact that with this definition, $f^\otimes(A,B)$ is not a square matrix of size $n$ anymore when $A$ and $B$ are square matrices of size $n$ (well, except in the trivial case where $n=1$) can seem weird and unsatisfying, but this objects naturally appear when one differentiates a smooth function of a matrix.

This work was supported by the TECHNION, Israel Institute of Tecnology, Haifa, Israel.

\section{definitions and notations}
As for defining functions in one variable applied to a single matrix (as done in \cite{GNT59}), we start with the simple case of polynomials.

\begin{defi}\label{polyM} Let $k\in\N^*$, and let $M_1,\ldots,M_k$ be endomorphisms of the finite dimensional $\C$-vector spaces $E_1,\ldots,E_k$ (their dimensions $n_1,\ldots,n_k$ may be different).
Let 
$$P(x_1,\ldots,x_k)=\sum_{\alpha\in F} c_\alpha x_1^{\alpha_1}\ldots x_k^{\alpha_k}$$
be a polynomial of $k$ variables (here $F$ is a finite subset of $\N^k$, $c_\alpha\in\C$ and the $\alpha_l$'s are the coordinates of $\alpha$).
We set:
$$P^\otimes(M_1,\ldots,M_k)\coleq\sum_{\alpha\in F}c_\alpha M_1^{\alpha_1}\otimes\ldots\otimes M_k^{\alpha_k}\in\bigotimes_{l=1}^kE_l\otimes E_l^*.$$
\end{defi}

\begin{rmk}\label{realpolyM}In this definition, the tensor products are implicitely taken over $\C$.
We can use a similar definition if $E_1,\ldots,E_k$ are $\Rl$-vector spaces and if $P$ has real coefficients, but with tensor products over $\Rl$.
In the cases when some of the $E_l$'s are $\Rl$-vector spaces and the other ones are $\C$-vector spaces, or if all the $E_l$'s are $\Rl$-vector spaces and $P$ has complex coefficients, we can complexify the real vector spaces by replacing the concerned $E_l$'s by $E'_l=\C\otimes_\Rl E_l$ and $M_l$ by its natural $\C$-linear extension $M'_l$ to $E'_l$, and use the classical version of definition \ref{polyM}.
\end{rmk}

\begin{rmk} We use the notation $P^\otimes(M_1,\ldots,M_k)$ and not $P(M_1,\ldots,M_k)$, because doing so would be confusing, and because the $\otimes$ sign reminds you it is a tensor of higher order.
For example, we have $((x,y)\mapsto x+y)^{\otimes}(A,B)=A\otimes I+I\otimes B\neq A+B$ in general (even in simple cases where $A=B$ or $B=0$).
\end{rmk}

\begin{lem}\label{conjug}
Let $k\in\N^*$, and let $M_1,\ldots,M_k$ be endomorphisms of the finite dimensional $\C$-vector spaces $E_1,\ldots,E_k$.
For $1\leq l\leq k$, let $A_{(l)}$ be an invertible $\C$-linear application from $E_l$ to another $\C$-vector space $F_l$ of same dimension than $E_l$.
We set $M'_l=A_{(l)}M_lA_{(l)}^{-1}\in L(F_l,F_l)$.
Let $P$ be a polynomial of $k$ variables.
Then we have:
\begin{multline*}P^\otimes(M'_1,\ldots,M'_k)^{i_1}{}_{j_1}\ldots^{i_k}{}_{j_k}=\\A_{(1)}{}^{i_1}{}_{i'_1}A_{(1)}^{-1}{}^{j'_1}{}_{j_1}\ldots A_{(k)}{}^{i_k}{}_{i'_k}A_{(k)}^{-1}{}^{j'_k}{}_{j_k}P^\otimes(M_1,\ldots,M_k)^{i'_1}{}_{j'_1}\ldots^{i'_k}{}_{j'_k}.
\end{multline*}
\end{lem}

\begin{proof}In the simple case where $P$ is a monomial, it follows from the fact that $(AMA^{-1})^n=AM^nA^{-1}$.
The more general case where $P$ is any polynomial easyly follows by linearity.
\end{proof}

Any complex square matrix can be put in a Jordan form by conjugation by an invertible matrix.
Let us see what one gets when all the matrices $M_1,\ldots,M_k$ are Jordan matrices.

\begin{xpl}\label{jordan}
Let $M_1,\ldots,M_k$ be Jordan matrices, i.e, for $1\leq l\leq k$, the matrix $M_l$ has the following form:
$$M_l=\left(
\begin{array}{cccc}J_{\lambda_{l1},r_{l1}}&0&\ldots&0\\
0&J_{\lambda_{l2},r_{l2}}&\ddots&\vdots\\
\vdots&\ddots&\ddots&0\\
0&\ldots&0&J_{\lambda_{lb_l},r_{lb_l}}
\end{array}\right),$$
where $b_l$ is the number of jordan blocks in $M_l$, $J_{\lambda,r}$ is the $r\times r$ square matrix
$$J_{\lambda,r}=\left(
\begin{array}{ccccc}\lambda&1&0&\ldots&0\\
0&\lambda&1&\ddots&\vdots\\
\vdots&\ddots&\ddots&\ddots&0\\
\vdots&\ddots&\ddots&\lambda&1\\
0&\ldots&\ldots&0&\lambda
\end{array}\right),$$
the $\lambda_{lm}$ are the eigenvalues of $M_l$ and the $r_{lm}$ are the sizes of the Jordan blocks of $M_l$.
Let $P$ be a polynomial on $k$ variables.

Let us choose the values of the indexes $i_1,\ldots,i_k$ and $j_1,\ldots,j_k$.
Then for each $1\leq l\leq k$, there exist $1\leq c_l,d_l\leq b_l$ such that $\sum_{m=1}^{c_l-1}r_{lm}<i_l\leq\sum_{m=1}^{c_l}r_{lm}$ and $\sum_{m=1}^{d_l-1}r_{lm}<j_l\leq\sum_{m=1}^{d_l}r_{lm}$.
Then the coefficient $P^\otimes(M_1,\ldots,M_k)^{i_1}{}_{j_1}\ldots^{i_k}{}_{j_k}$ is given by the following formulas:
\begin{itemize}
\item{}If for some $1\leq l\leq k$, one has $i_l>j_l$ or $c_l\neq d_l$, then the coefficient $P^\otimes(M_1,\ldots,M_k)^{i_1}{}_{j_1}\ldots^{i_k}{}_{j_k}$ is $0$.\\
\item{}If for all $1\leq l\leq k$, one has $i_l\leq j_l$ and $c_l=d_l$, then the coefficient $P^\otimes(M_1,\ldots,M_k)^{i_1}{}_{j_1}\ldots^{i_k}{}_{j_k}$ is $\left[\prod_{l=1}^k\frac{1}{(j_l-i_l)!}\partial_l^{j_l-i_l}\right]P(\lambda_{1c_1},\ldots,\lambda_{kc_k})$.
\end{itemize}
\end{xpl}

\begin{proof}Like in the proof of Lemma \ref{conjug}, we start with the simple case where $P$ is a monomial.
In this case, $P(x_1,\ldots,x_k)=x_1^{\alpha_1}\ldots x_k^{\alpha_k}$ and $P^\otimes(M_1,\ldots,M_k)^{i_1}{}_{j_1}\ldots^{i_k}{}_{j_k}=\prod_{l=1}^kM_l^{\alpha_l}{}^{i_l}{}_{j_l}$.

One easyly gets the expression given in Example \ref{jordan} for the coefficients of $P^\otimes(M_1,\ldots,M_k)$ by using the special form of the coefficients of powers of Jordan matrices, and the fact that $\left[\prod_{l=1}^k\frac{\partial}{\partial x_l}^{a_l}\right]\left(\prod_{l=1}^kf_l(x_l)\right)=\prod_{l=1}^kf_l^{(a_l)}(x_l)$.

The case of a more general polynomial $P$ trivially follows by linearity.
\end{proof}

\begin{prop}\label{zero}
Let $k\in\N^*$, and let $M_1,\ldots,M_k$ be square matrices with complex coefficients.
Let $P_1,\ldots,P_k$ be their minimal polynomials.

Then the set of polynomials $P$ of $k$ variables such that
$$P^\otimes(M_1,\ldots,M_k)=0$$
is the ideal generated by the polynomials $P_l(x_l)$, for $1\leq k\leq l$.
\end{prop}

\begin{proof}
There exist $M'_1,\ldots,M'_k$ Jordan matrices and $A_1,\ldots,A_k$ invertibles such that $M'_l=A_lM_lA_l^{-1}$.
According to Lemma \ref{conjug}, we have
$$P^\otimes(M_1,\ldots,M_k)=0\Leftrightarrow P^\otimes(M'_1,\ldots,M'_k)=0.$$
For $1\leq l\leq k$, we set $\lambda_{lm}$ the eigenvalues of $M_l$, and $r_{lm}$ their multiplicity as root of the minimal polynomial $P_l$ of $M_l$ (here the index $m$ varies from $1$ to the number $e_l$ of distinct eigenvalues of $M_l$).
Then according to example \ref{jordan}, $P^\otimes(M'_1,\ldots,M'_k)=0$ is equivalent to:
for all $(m_l)_{1\leq l\leq k},(j_l)_{1\leq l\leq k}$ satisfying $1\leq m_l\leq e_l$ and $0\leq j_l<r_{lm_l}$, one has
$$\left[\prod_{l=1}^k\partial_l^{j_l}\right]P(\lambda_{1m_1},\ldots,\lambda_{km_k})=0.$$

Assume that there exists $l$ such that $P(x_1,\ldots,x_k)=P_l(x_l)Q(x_1,\ldots,x_k)$, where $Q$ is a polynomial.
Then for any $1\leq m\leq e_l$ and any $0\leq j< m_{lm}$, $x_l-\lambda_{lm}$ divides $\partial_l^{j}P(x_1,\ldots,x_k)$, so it also divides all the derivatives of $\partial_l^{j}P(x_1,\ldots,x_k)$ with respect to all variables but $x_l$, thus they cancel when $x_l=\lambda_{lm}$.
Thus $P(M_1,\ldots,M_k)=0$.
By linearity, for any polynomial $P$ in the ideal generated by $P_1(x_1),\ldots,P_k(x_k)$, we have $P(M_1,\ldots,M_k)=0$.

Conversely, let $P$ be a polynomial such that $P(M_1,\ldots,M_k)=0$.
There exists two polynomials $Q$ and $R$ such that $P=Q+R$, $Q$ belongs to the ideal generated by $P_1(x_1),\ldots,P_k(x_k)$ and the degree of $R$ with respect to the variable $x_l$ is lower than $\deg(P_l)$ (in the sense that any monomial $x_1^{\alpha_1}\ldots x_k^{\alpha_k}$ in $R$ satisfies $\alpha_l<\deg(P_l)$).
This fact can be shown by performing Euclidean division by the Gr\"obner basis $P_1(x_1),\ldots,P_k(x_k)$ (this is a Gr\"obner basis for any monomial order) in the space of polynomials on $k$ variables.
For any $1\leq l\leq k$, any $1\leq m\leq e_l$ and any $0\leq j<r_{lm}$, there exists a polynomial of one variable $P_{lmj}$ of degree less than $\deg(P_l)$ such that for any $1\leq m'\leq e_l$ and any $0\leq j'<r_{lm'}$, one has $P_{lmj}^{(j')}(\lambda_{lm'})=\car_{m=m',j=j'}$ (this follows from the classical theory of Lagrange--Sylvester interpolation polynomials).
For any $m_1,\ldots,m_k$ and $j_1,\ldots,j_k$ such that $1\leq m_l\leq e_l$ and $0\leq j_l<r_{lm_l}$, we set $P_{m_1j_1\ldots m_kj_k}(x_1,\ldots x_k)=\prod_{l=1}^kP_{lm_lj_l}(x_l)$.

Let us consider the linear map $\nu$ which associates to any polynomial $S(x_1,\ldots,x_k)$ of $k$ variables such that its degree with respect to $x_l$ is less than $\deg(P_l)$, the following $\prod_{l=1}^k\deg(P_l)$ values:
for any $m_1,\ldots,m_k$ and $j_1,\ldots,j_k$ satisfying $1\leq m_l\leq e_l$ and $0\leq j_l<r_{lm_l}$, we set
$$\nu_{m_1j_1\ldots m_kj_k}(S)\coleq\left[\prod_{l=1}^k \partial_l^{j_l}\right]S(\lambda_{1m_1},\ldots,\lambda_{km_k}).$$
Then we have $\nu_{m_1j_1\ldots m_kj_k}(P_{m'_1j'_1\ldots m'_kj'_k})=1$ if $m_l=m'_l$ and $j_l=j'_l$ for every $1\leq l\leq k$ and $\nu_{m_1j_1\ldots m_kj_k}(P_{m'_1j'_1\ldots m'_kj'_k})=0$ otherwise, so the linear map $\nu$ is surjective.
So because of the equality of the (finite) dimensions of its domain and its image, $\nu$ is also injective.
So since $\nu(R)=0$, we have $R=0$ and thus $P=Q$ belongs to the ideal generated by the $P_l(x_l)$.
\end{proof}

So one does see from Proposition \ref{zero} that the dependancy in $P$ of $P(M_1,\ldots,M_k)$ only relies on the values of $P$ and of some of its derivatives at the eigenvalues of the matrices $M_l$, $1\leq l\leq k$.
This justifies the following definition for more general functions than polynomials.
\begin{defi}\label{funcM}
Let $k\in\N^*$, and let $M_1,\ldots,M_k$ be endomorphisms of the finite dimensional $\C$-vector spaces $E_1,\ldots,E_k$.
For $1\leq l\leq k$, let $e_l$ be the number of different eigenvalues of $M_l$, $\lambda_{lm}$ be the eigenvalues of $M_l$, for $1\leq m\leq e_l$, and $r_{lm}$ be the multiplicity of the root $\lambda_{lm}$ in the minimal polynomial of $M_l$.
Let $f$ be a function of $k$ complex variables.
Assume that for any $k$-uple of integers $(m_1,\ldots,m_k)$ such that $1\leq m_l\leq e_l$ for $1\leq l\leq k$, the function $f(x_1,\ldots,x_k)$ is holomorphic with respect to the variables $x_l$ with $l$ satisfying $r_{lm_l}\geq 2$ near $(\lambda_{1m_1},\ldots,\lambda_{km_k})$.
That is, if we denote by $l_1,\ldots,l_p$ the indexes such that $r_{l_q}\geq 2$ for $1,\leq q\leq p$, the function
$$g:\left\{\begin{array}{rcl}\C^p&\mapsto&\C\\
(y_1,\ldots,y_p)&\rightarrow&f\left((\lambda_{1m_1},\ldots,\lambda_{km_k})+\sum_{q=1}^py_qv_{l_q}\right)\end{array}\right.$$
admits a continuous $\C$-linear differential on a neighborhood of $0$, where $v_l$ is the $l$-th vector of the canonical basis of $\C^k$.
Then we set:
$$f^\otimes(M_1,\ldots,M_k)\coleq P^\otimes(M_1,\ldots,M_k),$$
with $P$ any complex polynomial of $k$ variables such that for all $(m_1,\ldots,m_k)$ and $(j_1,\ldots,j_k)$ such that $1\leq m_l\leq e_l$ and $0\leq j_l<r_{lm_l}$ for $1\leq l\leq k$, we have:
$$\left[\prod_{l=1}^k\partial_l^{j_l}\right]f(\lambda_{1m_1},\ldots,\lambda_{km_k})=\left[\prod_{l=1}^k\partial_l^{j_l}\right]P(\lambda_{1m_1},\ldots,\lambda_{km_k}).$$
\end{defi}
\begin{rmk}The function $f$ does not really need to be defined on the whole $\C^k$, but only on the points whose coordinates are the eigenvalues of the $M_l$'s and on some neighborhoods of those points intersected with some affine subspaces.
\end{rmk}
\begin{rmk}If one wants to generalize this definition in a real framework, as said in Remark \ref{realpolyM}, one can not just assume $f$ is a real function of real variables, because real matrices may have non-real eigenvalues.

The right thing to do is to chose $f$ such that $f(\bar{z_1},\ldots,\bar{z_k})=\overline{f(z_1,\ldots,z_k)}$ to be sure there exists a polynomial $P$ with real coefficients satisfying all the equalities required.
\end{rmk}
\begin{rmk}The polynomial $P$ can be chosen to be $$\sum \left(\left[\prod_{l=1}^k\partial_l^{j_l}\right]f(\lambda_{1m_1},\ldots,\lambda_{km_k})\right)P_{m_1j_1\ldots m_kj_k}(x_1,\ldots,x_k),$$ with the same notation as in the proof of Proposition \ref{zero}.
In this case, we say $P$ is the Lagrange--Sylvester interpolation polynomial of $f$ associated to the product of multisets $\prod_{l=1}^k\Root(P_l)$, where $\Root(P_l)$ is the multiset of the roots of $P_l$, counted with their multiplicity.
\end{rmk}

\section{Some rules for computations}
In this section, we show that $f^\otimes(M_1,\ldots,M_k)$ behaves like fonctions of several variables do, and we show that some contractions of $f^\otimes(M_1,\ldots,M_k)$ can have a simplified expression.
Finally, we show the main result of this paper, which is the expression of the derivative of $f^\otimes(M_1,\ldots,M_k)$ with respect to the matrices $M_l$.

The tensor $f^\otimes(M_1,\ldots,M_k)$ applied to eigenvectors of $M_1,\ldots,M_k$ has a simplified expression.
This allows to give an alternative expression for $f^\otimes(M_1,\ldots,M_k)$ when $M_1,\ldots,M_k$ are all diagonalizable.

\begin{prop}\label{evapp}Let $k\in\N^*$, and let $M_1,\ldots,M_k$ be endomorphisms of the finite dimensional $\C$-vector spaces $E_1,\ldots,E_k$.
For $1\leq l\leq k$, let $u_l\in E_l$ be an eigenvector of $M_l$ for the eigenvalue $\lambda_l$.
then we have
$$f^\otimes(M_1,\ldots,M_k)^{i_1}{}_{j_1}{}\ldots^{i_k}{}_{j_k}(u_1)^{j_1}\ldots(u_k)^{j_k}=f(\lambda_1,\ldots,\lambda_k)(u_1)^{i_1}\ldots(u_k)^{i_k}.$$
\end{prop}

\begin{proof}In the simple case where $f$ is a monomial, we have $f(x_1,\ldots,x_k)=x_1^{\alpha_1}\ldots x_k^{\alpha_k}$, with $\alpha_l\in\N$ for $1\leq l\leq k$.
We get
\begin{align*}f^\otimes(M_1,\ldots,M_k)^{i_1}{}_{j_1}{}\ldots^{i_k}{}_{j_k}(u_1)^{j_1}\ldots(u_k)^{j_k}&=\left(M_1^{\alpha_1}\right)^{i_1}{}_{j_1}\ldots\left(M_k^{\alpha_k}\right)^{i_k}{}_{j_k}(u_1)^{j_1}\ldots(u_k)^{j_k}\\
{}&=\lambda_1^{\alpha_1}(u_1)^{i_1}\ldots\lambda_k^{\alpha_k}(u_k)^{i_k}\\
{}&=f(\lambda_1,\ldots,\lambda_k)(u_1)^{i_1}\ldots(u_k)^{i_k}.\end{align*}

By linearity, this extends to the case where $f$ is a polynomial.
For the more general case, we have $f^\otimes(M_1,\ldots,M_k)=P^\otimes(M_1,\ldots,M_k)$ where $P$ is a polynomial of $k$ variables satisfying the conditions described in Definition \ref{funcM}. In particular, $P$ has the same values $f$ has on the eigenvalues of $M_1,\ldots,M_k$, so the desired equality holds for a general function $f$.
\end{proof}

\begin{rmk}In the case where $M_1,\ldots,M_k$ are all diagonalizable, for all $1\leq l\leq k$ we can take a basis of $E_l$ made of eigenvectors of $M_l$.
Let us denote by $u_{lm}$, for $1\leq m\leq\mathrm{dim}(E_l)=d_l$ the vectors of this basis, by $\lambda_{lm}$ the corresponding eigenvalue and by $u^\ast_{lm}\in E^\ast_l$ the corresponding vector of the dual basis.

Then the family of tensors $\bigotimes_{l=1}^ku_{lm_l}\otimes u^\ast_{ln_l}$, for $1\leq m_l\leq d_l$ and $1\leq n_l\leq d_l$, is a basis of $\bigotimes_{l=1}^kE_l\otimes E^\ast_l$.
According to Proposition \ref{evapp}, the decomposition of the tensor $f^\otimes(M_1,\ldots,M_k)$ in this basis can only be
$$f^\otimes(M_1,\ldots,M_k)=\sum_{\forall 1\leq l\leq k, 1\leq m_l\leq d_l}f(\lambda_{1m_1},\ldots,\lambda_{km_k})\bigotimes_{l=1}^ku_{lm_l}\otimes u^\ast_{lm_l}.$$
\end{rmk}

One of the simplest properties of $f^\otimes(M_1,\ldots,M_k)$ is its linearity with respect to $f$.

\begin{prop}Let $k\in\N^*$, and let $M_1,\ldots,M_k$ be endomorphisms of the finite dimensional $\C$-vector spaces $E_1,\ldots,E_k$.
Let $\lambda$ and $\mu$ be two complex numbers, and $f$ and $g$ be two functions from $\C^k$ to $\C$, regular enough such that $f^\otimes(M_1,\ldots,M_k)$ and $g^\otimes(M_1,\ldots,M_k)$ can be defined.
Then the function $\lambda f+\mu g$ has the same regularity, and we have:
$$(\lambda f+\mu g)^\otimes(M_1,\ldots,M_k)=\lambda f^\otimes(M_1,\ldots,M_k)+\mu g^\otimes(M_1,\ldots,M_k).$$
\end{prop}

\begin{proof}If $P$ and $Q$ are interpolation polynomials of $f$ and $g$ as required in Definition \ref{funcM} to compute $f^\otimes(M_1,\ldots,M_k)$ and $g^\otimes(M_1,\ldots,M_k)$, then $\lambda P+\mu Q$ is an interpolation polynomial of $\lambda f+\mu g$ because partial derivatives are linear.

So we have
\begin{align*}(\lambda f+\mu g)^\otimes(M_1,\ldots,M_k)&=(\lambda P+\mu Q)^\otimes(M_1,\ldots,M_k)\\
&=\lambda P^\otimes(M_1,\ldots,M_k)+\mu Q^\otimes(M_1,\ldots,M_k)\\
&=\lambda f^\otimes(M_1,\ldots,M_k)+\mu g^\otimes(M_1,\ldots,M_k),\end{align*}
where the second equality trivially follows from Definition \ref{polyM}.
\end{proof}

Another trivial property is the nice behaviour of $f^\otimes(M_1,\ldots,M_k)$ with respect to transpositions.
\begin{prop}\label{transp}Let $k\in\N^*$, and let $M_1,\ldots,M_k$ be endomorphisms of the finite dimensional $\C$-vector spaces $E_1,\ldots,E_k$.
Let $f$ be a function from $\C^k$ to $\C$ such that $f^\otimes(M_1,\ldots,M_k)$ is well defined.
Let $1\leq l\leq k$.
The transposition of $M_l$ is the endomorphism of $E_l^*$ defined by
$$\langle M_l^T(\zeta),v\rangle=\langle\zeta,M_l(v)\rangle,$$
for any $\zeta$ in $E_l^*$ and $v\in E_l$.
More explicitely, we can write $(M_l^T)_{j_l}{}^{i_l}=(M_l)^{i_l}{}_{j_l}$.

Then we have:
\begin{multline*}f^\otimes(M_1,\ldots,M_{l-1},M_l^T,M_{l+1},\ldots,M_k)^{i_1}{}_{j_1}\ldots^{i_{l-1}}{}_{j_{l-1}j_l}{}^{i_li_{l+1}}{}_{j_{l+1}}\ldots^{i_k}{}_{j_k}=\\f^\otimes(M_1,\ldots,M_k)^{i_1}{}_{j_1}\ldots^{i_k}{}_{j_k}.\end{multline*}
\end{prop}

\begin{rmk}In particular, if the matrix of $M_l$ is symmetric in some basis $\mathcal{B}$ of $E_l$, then the coefficients of $f^\otimes(M_1,\ldots,M_k)$ in a product basis where the one chosen for $E_l$ is $\mathcal{B}$ are invariant by swapping the corresponding indexes.
\end{rmk}

\begin{proof}The matrices $M_l$ and $M_l^T$ have the same eigenvalues and the same minimal polynomial, thus if $P$ is a suitable interpolation polynomial so that $f^\otimes(M_1,\ldots,M_k)=P^\otimes(M_1,\ldots,M_k)$, then we have $f^\otimes(M_1,\ldots,M_{l-1},M_l^T,M_{l+1},\ldots,M_k)=P^\otimes(M_1,\ldots,M_{l-1},M_l^T,M_{l+1},\ldots,M_k)$.
Proposition \ref{transp} for polynomials trivially follows from the fact that $(M^T)^a=(M^a)^T$.\end{proof}

The tensor $f^\otimes(M_1,\ldots,M_k)$ can also be seen as an endomorphism on the space $E_1\otimes E_2\otimes\ldots\otimes E_k$.
This allows to take products of such tensors, or to apply functions of several variables on them.

\begin{thm}\label{product}Let $M_1,\ldots,M_k$ be endomorphisms of finite dimensional $\C$-vector spaces.
Let $f_1$ and $f_2$ be two functions from $\C^k$ to $\C$.
As in Definition \ref{funcM}, we set $e_l$ the number of different eigenvalues of $M_l$, for $1\leq l\leq k$, $\lambda_{lm}$ the eigenvalues of $M_l$ and $r_{lm}$ their multiplicity as roots of the minimal polynomial of $M_l$, for $1\leq m\leq e_l$.
Assume that $f_1$ and $f_2$ are holomorphic with respect to all the variables $x_l$ such that $r_{lm_l}\geq 2$ near $(\lambda_{1m_1},\ldots,\lambda_{km_k})$.
Then we can define $\bar{M}_1=f_1^\otimes(M_1,\ldots,M_k)$ and $\bar{M}_2=f_2^\otimes(M_1,\ldots,M_k)$ and see them as endomorphisms of $E_1\otimes\ldots\otimes E_k$.
Then we have
$$\bar{M}_1\bar{M}_2=g^\otimes(M_1,\ldots,M_k),$$
with $g(x_1,\ldots,x_k)=f_1(x_1,\ldots,x_k)f_2(x_1,\ldots,x_k)$.
\end{thm}

\begin{rmk}\label{commut1}In particular, $\bar{M}_1$ and $\bar{M}_2$ commute, since one does get the same function $g$ by swapping $f_1$ and $f_2$.
\end{rmk}

\begin{proof} We first notice that the function $g$ is regular enough so that $g^\otimes(M_1,\ldots,M_k)$ is well defined, so the statement of Proposition \ref{product} has a sense.

Let $P_1$ and $P_2$ be suitable interpolation polynomials of $f_1$ and $f_2$, in the sense that all the partial derivatives at $(\lambda_{1m_1},\ldots,\lambda_{km_k})$ such that one derives less than $r_{lm_l}$ times with respect to $x_l$ are equal for $f_i$ and $P_i$.

Then according to Definition \ref{funcM}, one has $f_i^\otimes(M_1,\ldots,M_k)=P_i^\otimes(M_1,\ldots,M_k)$.
The polynomials $P_1$ and $P_2$ can be written as
$$P_i=\sum_{\alpha\in F_i}a_{i\alpha}x_1^{\alpha_1}\ldots x_k^{\alpha_k},$$
with $F_i$ a finite subset of $\N^k$ and $a_{i\alpha}$ the complex coefficients of the polynomial $P_i$.
Then we have:
\begin{align*}(\bar{M}_1\bar{M}_2)^{i_1}{}_{j_1}\ldots^{i_k}{}_{j_k}&=(\bar{M}_1)^{i_1}{}_{m_1}\ldots^{i_k}{}_{m_k}(\bar{M}_2)^{m_1}{}_{j_1}\ldots^{m_k}{}_{j_k}\\
&=\sum_{\alpha\in F_1}a_{1\alpha} (M_1^{\alpha_1})^{i_1}{}_{m_1}\ldots(M_k^{\alpha_k})^{i_k}{}_{m_k}\sum_{\beta\in F_2}a_{2\beta} (M_1^{\beta_1})^{m_1}{}_{j_1}\ldots(M_k^{\beta_k})^{m_k}{}_{j_k}\\
&=\sum_{\alpha\in F_1,\beta\in F_2}a_{1\alpha}a_{2\beta}(M_1^{\alpha_1+\beta_1})^{i_1}{}_{j_1}\ldots(M_k^{\alpha_k+\beta_k})^{i_k}{}_{j_k}\\
&=(P_1P_2)^\otimes(M_1,\ldots,M_k)^{i_1}{}_{j_1}\ldots^{i_k}{}_{j_k}.
\end{align*}
Finally, we have
\begin{multline*}\partial_1^{a_1},\ldots,\partial_k^{a_k}(P_1P_2)(x_1,\ldots,x_k)=\\
\sum_{\substack{0\leq b_1\leq a_1\\\ldots\\0\leq b_k\leq a_k}}\binom{a_1}{b_1}\ldots\binom{a_k}{b_k}\partial_1^{b_1}\ldots\partial_k^{b_k}P_1(x_1,\ldots,x_k)\partial_1^{a_1-b_1}\ldots\partial_k^{a_k-b_k}P_2(x_1,\ldots,x_k).\end{multline*}
the same formula holds when we replace $P_1$ with $f_1$, $P_2$ with $f_2$, $P_1P_2$ with $g$, and $x_l$ with $\lambda_{lm_l}$, for all $k$-tuple $(m_1,\ldots,m_k)$ such that $1\leq m_l\leq e_l$, provided $0\leq a_l<r_{lm_l}$ for all $1\leq l\leq k$.
Thus we have
$$\partial_1^{a_1}\ldots\partial_k^{a_k}g(\lambda_{1m_1},\ldots,\lambda_{km_k})=\partial_1^{a_1}\ldots\partial_k^{a_k}(P_1P_2)(\lambda_{1m_1},\ldots,\lambda_{km_k}),$$
for all $k$-tuple $(a_1,\ldots,a_k)$ such that $0\leq a_l<r_{lm_l}$.
So we get
$$\bar{M}_1\bar{M}_2=(P_1P_2)^\otimes(M_1,\ldots,M_k)=g^\otimes(M_1,\ldots,M_k),$$
as stated.\end{proof}

\begin{thm}\label{compose} Let $r\in\N^*$, and $k_q\in\N^*$ for $1\leq q\leq r$.
Let $E_{ql}$ be a family of finite dimensional $\C$-vector spaces, where $1\leq l\leq r$ and $1\leq l\leq k_q$, and let $M_{ql}$ be and endomorphism of $E_{ql}$.

We set $e_{ql}$ the number of different eigenvalues of $M_{ql}$, $\lambda_{qlm}$ these eigenvalues, where $1\leq m\leq e_{ql}$, and $r_{qlm}$ the multiplicity of the root $\lambda_{qlm}$ in the minimal polynomial of $M_{ql}$.
For $1\leq q\leq r$, let $f_q$ be a function from $\C^{k_q}$ to $\C$, such that $f_q$ is holomorphic with respect to all the variables $x_l$ such that $r_{qlm_l}\geq2$ near $(\lambda_{q1m_1},\ldots,\lambda_{qlm_l},\ldots,\lambda_{qk_qm_{k_q}})$, for every $k_q$-tuple $(m_1,\ldots,m_{k_q})$ with $1\leq m_l\leq e_{ql}$.

Let $\bar{e}_q$ be the number of different values that $f_q(\lambda_{q1m_1},\ldots,\lambda_{qk_qm_{k_q}})$ takes, and $\mu_{qp}$ be these values for $1\leq p\leq\bar{e}_q$.

We set $\bar{r}_{qp}\coleq\max\{1+\sum_{l=1}^{k_q}(r_{qlm_l}-1),1\leq m_l\leq e_{ql}|f_q(\lambda_{q1m_1},\ldots,\lambda_{qk_qm_{k_q}})=\mu_{qp}\}$.

We set $\bar{M}_q=f_q^\otimes(M_{q1},\ldots,M_{qk_q})$, seen as an endomorphism of $\bar{E}_{q}\coleq E_{q1}\otimes\ldots\otimes E_{qk_q}$.

Let $g$ be a function from $\C^r$ to $\C$, such that for every $r$-tuple $(p_1,\ldots,p_r)$, $g$ is holomorphic with respect to all the variables $x_q$ such that $\bar{r}_{qp_q}\geq2$, near the point $(\mu_{1p_1},\ldots,\mu_{rp_r})$.

Then for all $1\leq q\leq k$, and all $1\leq p\leq \bar{e}_q$, the $\mu_{qp}$'s are the eigenvalues of $\bar{M}_q$, and the multiplicity of the root $\mu_{qp}$ in the minimal polynomial of $\bar{M}_q$ is at most $\bar{r}_qp$.

So $g^\otimes(\bar{M}_1,\ldots,\bar{M}_r)$ is well defined and furthermore, we have:
$$g^\otimes(\bar{M}_1,\ldots,\bar{M}_r)=h(M_{11},\ldots,M_{rk_r}),$$
where $h$ is the function of $\sum_{q=1}^rk_q$ variables defined by
$$h(x_{11},\ldots,x_{rk_r})\coleq g(f_1(x_{11},\ldots,x_{1k_1}),\ldots,f_r(x_{r1},\ldots,x_{rk_r})).$$
\end{thm}

\begin{proof}Let us prove that the $\mu_{qp}$'s are the eigenvalues of $\bar{M}_q$ and have multiplicity at most $\bar{r}_{pq}$ in the minimal polynomial of $\bar{M}_q$.

Let $1\leq q \leq r$.
Let $P(x_1,\ldots,x_{k_q})$ be a suitable interpolation polynomial of $f_q$, such that $\bar{M}_q=P^\otimes(M_{q1},\ldots,M_{qk_q})$.
For any $m_1,\ldots,m_{k_q}$, if $v_1,\ldots,v_{k_q}$ are eigenvectors of $M_{q1},\ldots,M_{qk_q}$ for the eigenvalues $\lambda_{q1m_1},\ldots,\lambda_{qk_qm_{k_q}}$, then $v_1\otimes\ldots\otimes v_{k_q}$ is an eigenvector of $\bar{M}_q$ for the eigenvalue $P(\lambda_{q1m_1},\ldots,\lambda_{qk_qm_{k_q}})=f_q(\lambda_{q1m_1},\ldots,\lambda_{qk_qm_{k_q}})$, so the $\mu_{qp}$'s are eigenvalues of $\bar{M}_q$.
Let $Q(x)=\prod_{p=1}^{\bar{e}_q}(x-\mu_{qp})^{\bar{r}_{qp}}$.
It remains to check that $Q(\bar{M}_q)=0$.

We set $R(x_1,\ldots,x_{k_q})=Q(P(x_1,\ldots,x_{k_q}))$.
Then it easily follows from Theorem \ref{product} that $Q(\bar{M}_q)=R^\otimes(M_{q1},\ldots,M_{qk_q})$.
For $1\leq l\leq k_q$, let $1\leq m_l\leq e_{ql}$, and $0\leq j_l<r_{qlm_l}$.
One has, according to the Faa di Bruno formula,
$$\left[\prod_{l=1}^{k_q}\partial_l^{j_l}\right]R(\lambda_{q1m_1},\ldots,\lambda_{qk_qm_{k_q}})=\sum_{j=0}^{\sum_{l=1}^{k_q}j_l}Q^{(j)}(P(\lambda_{q1m_1},\ldots,\lambda_{qk_qm_{k_q}}))A_j,$$
where the term $A_j$ can be written as
$$A_j=\sum_{\substack{n\leq j\\(i_1,\ldots,i_n)\in(\N^{k_q})^n\\0\prec i_1\prec i_2\prec\ldots\prec i_n\\(a_1,\ldots,a_n)\in\N^{*n}\\a_1+\ldots+a_n=j\\a_1i_1+\ldots+a_ni_n=(j_1,\ldots,j_{k_q})}}\left(\frac{\prod_{l=1}^{k_q}j_l!}{\prod_{k=1}^{n}a_k!\prod_{l=1}^{k_q}(i_{kl}!)^{a_k}}\right)\prod_{k=1}^n\left(\left[\prod_{l=1}^{k_q}\partial_l^{i_{kl}}\right]P(\lambda_{q1m_1},\ldots,\lambda_{qk_qm_{k_q}})\right)^{a_k},$$
where $\preceq$ can be any total order on $\N^{k_q}$ such that $0$ is its minimal element (one can take the lexicographical order, for example).

But actually the value of $A_j$ does not matter, since $Q^{(j)}(P(\lambda_{q1m_1},\ldots,\lambda_{qk_qm_{k_q}}))=0$, because $P(\lambda_{q1m_1},\ldots,\lambda_{qk_qm_{k_q}})$ is equal to some $\mu_{qp}$ such that $\bar{r}_{qp}>j$.
Thus, we have $\left[\prod_{l=1}^{k_q}\partial_l^{j_l}\right]R(\lambda_{q1m_1},\ldots,\lambda_{qk_qm_{k_q}})=0$ and then by Proposition \ref{zero}, one has $R^\otimes(M_{q1},\ldots,M_{qk_q})$, and hence $Q(\bar{M}_q)=0$ as stated.

Now let $P$ be a suitable interpolation polynomial of $g$, i.e. such that
$$\left[\prod_{q=1}^r\partial_q^{j_q}\right]P(\mu_{1p_1},\ldots,\mu_{rp_r})=\left[\prod_{q=1}^r\partial_q^{j_q}\right]g(\mu_{1p_1},\ldots,\mu_{rp_r}),$$
for all $p_1,\ldots,p_r$ and $j_1,\ldots,j_r$ satisfying $1\leq p_q\leq \bar{e}_q$ and $0\leq j_q<\bar{r}_{qp_q}$.
Then we have $P^\otimes(\bar{M}_1,\ldots,\bar{M}_r)=g^\otimes(\bar{M}_1,\ldots,\bar{M}_r)$.
It easily follows from Definition \ref{polyM} and Theorem \ref{product} that
$$P^\otimes(\bar{M}_1,\ldots,\bar{M}_r)=H^\otimes(M_{11},\ldots,M_{rk_r}),$$
where $H$ is the function of $\sum_{q=1}^rk_q$ variables defined by
$$H(x_{11},\ldots,x_{rk_r})=P(f_1(x_{11},\ldots,x_{1k_1}),\ldots,f_r(x_{r1},\ldots,x_{rk_r})).$$
To complete the proof of the theorem, it remains to check that
$$\left[\prod_{q=1}^r\prod_{l=1}^{k_q}\partial_{ql}^{j_{ql}}\right]H(\lambda_{11m_{11}},\ldots,\lambda_{rk_rm_{rk_r}})=\left[\prod_{q=1}^r\prod_{l=1}^{k_q}\partial_{ql}^{j_{ql}}\right]h(\lambda_{11m_{11}},\ldots,\lambda_{rk_rm_{rk_r}}),$$
for all $1\leq m_{ql}\leq e_{ql}$ and $0\leq j_{ql}<r_{qlm_{ql}}$.
This is easily done by using the Faa di Bruno formula and the fact that
$$\left[\prod_{q=1}^r\partial_q^{j_q}\right]P(\mu_{1p_1},\ldots,\mu_{rp_r})=\left[\prod_{q=1}^r\partial_q^{j_q}\right]g(\mu_{1p_1},\ldots,\mu_{rp_r}),$$
for all $1\leq p_q\leq\bar{e}_q$ and $0\leq j_q<\bar{r}_{qp_q}$.
\end{proof}

In some cases, contractions of the tensor $f^\otimes(M_1,\ldots,M_k)$ have a simplified expression.

\begin{thm}\label{contr1}Let $M_1,\ldots,M_k$ be endomorphisms of finite dimensional $\C$-vector spaces.
Let $f$ be a function from $\C^k$ to $\C$, such that $f^\otimes(M_1,\ldots,M_k)$ is well defined.
Let $1\leq p\leq k$.
Then we have the following equality:
\begin{multline*}f^\otimes(M_1,\ldots,M_k)^{i_1}{}_{j_1}\ldots^{i_{p-1}}{}_{j_{p-1}}{}^{i_p}{}_{i_p}{}^{i_{p+1}}{}_{j_{p+1}}\ldots^{i_k}{}_{j_k}=\\
g^\otimes(M_1,\ldots,M_{p-1},M_{p+1},\ldots,M_k)^{i_1}{}_{j_1}\ldots^{i_{p-1}}{}_{j_{p-1}}{}^{i_{p+1}}{}_{j_{p+1}}\ldots^{i_k}{}_{j_k},\end{multline*}
where $g$ is the function of $k-1$ variables defined by
$$g(x_1,\ldots,x_{p-1},x_{p+1},\ldots,x_k)=\sum_{m=1}^{e_p}s_{pm}f(x_1,\ldots,x_{p-1},\lambda_{pm},x_{p+1},\ldots,x_k),$$
where the $\lambda_{pm}$, for $1\leq m\leq e_p$ are all the different eigenvalues of $M_p$ and $s_{pm}$ is the multiplicity of the eigenvalue $\lambda_{pm}$, i.e. its multiplicity as root of the characteristic polynomial of $M_p$.
\end{thm}

\begin{proof}In the simple case where $f(x_1,\ldots,x_k)$ is a monomial, the equality of Theorem \ref{contr1} easily follows from the classical fact that $\Tr(M^a)=\sum_{\lambda\textrm{ eigenvalue of }M}\lambda^a$ (where the same $\lambda$ appears several times in the sum if it is a multiple eigenvalue of $M$) for any square matrice $M$.
By linearity, the equality of Theorem \ref{contr1} extends to the case where $f$ is a polynomial.

For a more general $f$, we set $P$ a polynomial such that $\left[\prod_{l=1}^k\partial_l^{a_l}\right](f-P)(\lambda_{1m_1},\ldots,\lambda_{km_k})=0$ for any sequences $m_1,\ldots,m_k$ and $a_1,\ldots,a_k$ such that $1\leq m_l\leq e_l$ and $0\leq a_l<r_{lm_l}$, with $e_l$ the number of different eigenvalues of $M_l$, $\lambda_{lm}$ these eigenvalues for $1\leq m\leq e_l$ and $r_{lm}$ their multiplicities in the minimal polynomial of $M_l$.
Then we have $f^\otimes(M_1,\ldots,M_k)=P^\otimes(M_1,\ldots,M_k)$.
Thus we get
\begin{multline*}f^\otimes(M_1,\ldots,M_k)^{i_1}{}_{j_1}\ldots^{i_{p-1}}{}_{j_{p-1}}{}^{i_p}{}_{i_p}{}^{i_{p+1}}{}_{j_{p+1}}\ldots^{i_k}{}_{j_k}=\\
P^\otimes(M_1,\ldots,M_k)^{i_1}{}_{j_1}\ldots^{i_{p-1}}{}_{j_{p-1}}{}^{i_p}{}_{i_p}{}^{i_{p+1}}{}_{j_{p+1}}\ldots^{i_k}{}_{j_k}=\\
Q^\otimes(M_1,\ldots,M_{p-1},M_{p+1},\ldots,M_k)^{i_1}{}_{j_1}\ldots^{i_{p-1}}{}_{j_{p-1}}{}^{i_{p+1}}{}_{j_{p+1}}\ldots^{i_k}{}_{j_k},
\end{multline*}
with
$$Q(x_1,\ldots,x_{p-1},x_{p+1},\ldots,x_k)=\sum_{m=1}^{e_p}s_{pm}P(x_1,\ldots,x_{p-1},\lambda_{pm},x_{p+1},\ldots,x_k).$$

To conclude the proof, it remains to check that
$$\left.\left[\prod_{\substack{1\leq l\leq k\\l\neq p}}\left(\frac{\partial}{\partial x_l}\right)^{a_l}\right](g-Q)(x_1,\ldots,x_{p-1},x_{p+1},\ldots,x_k)\right|_{(x_1,\ldots,x_k)=(\lambda_{1m_1},\ldots,\lambda_{km_k})}=0.$$
This quantity is just
$$\sum_{m=1}^{e_p}\left[\prod_{\substack{1\leq l\leq k\\l\neq p}}\left(\frac{\partial}{\partial x_l}\right)^{a_l}\right](f-P)(\lambda_{1m_1},\ldots,\lambda_{p-1m_{p-1}},\lambda_{pm},\lambda_{p+1m_{p+1}},\ldots,\lambda_{km_k}),$$
which is trivially $0$.

So $Q^\otimes(M_1,\ldots,M_{p-1},M_{p+1},\ldots,M_{k})=g^\otimes(M_1,\ldots,M_{p-1},M_{p+1},\ldots,M_{k})$ and the theorem is proved.
\end{proof}

\begin{thm}\label{contr2}Let $M_1,\ldots,M_k$ be endomorphisms of finite dimensional $\C$-vector spaces $E_1,\ldots,E_k$.
Assume that for two indexes $1\leq p<q\leq k$, we have $E_p=E_q$ and $M_p=M_q$.
Let $f$ be a function from $\C^k$ to $\C$, such that $f^\otimes(M_1,\ldots,M_k)$ is well defined.
Then we have:
$$\begin{array}{l}f^\otimes(M_1,\ldots,M_k)^{i_1}{}_{j_1}\ldots^{i_{p-1}}{}_{j_{p-1}}{}^i{}_a{}^{i_{p+1}}{}_{j_{p+1}}\ldots^{i_{q-1}}{}_{j_{q-1}}{}^a{}_j{}^{i_{q+1}}{}_{j_{q+1}}\ldots^{i_k}{}_{j_k}=\\
f^\otimes(M_1,\ldots,M_k)^{i_1}{}_{j_1}\ldots^{i_{p-1}}{}_{j_{p-1}}{}^a{}_j{}^{i_{p+1}}{}_{j_{p+1}}\ldots^{i_{q-1}}{}_{j_{q-1}}{}^i{}_a{}^{i_{q+1}}{}_{j_{q+1}}\ldots^{i_k}{}_{j_k}=\\
g^\otimes(M_1,\ldots,M_{q-1},M_{q+1},\ldots,M_k)^{i_1}{}_{j_1}\ldots^{i_{p-1}}{}_{j_{p-1}}{}^i{}_j{}^{i_{p+1}}{}_{j_{p+1}}\ldots^{i_{q-1}}{}_{j_{q-1}}{}^{i_{q+1}}{}_{j_{q+1}}\ldots^{i_k}{}_{j_k},\end{array}$$
where the function $g$ of $k-1$ variables is defined by
$$g(x_1,\ldots,x_{q-1},x_{q+1},\ldots,x_k)=f(x_1,\ldots,x_{q-1},x_p,x_{q+1},\ldots,x_k).$$
\end{thm}

\begin{proof}We proceed as in the proof of Theorem \ref{contr1}.
If $f$ is a monomial, the result easily follows from the trivial fact that $M^aM^b=M^bM^a=M^{a+b}$ for any square matrix $M$.
So by linearity, it also holds for polynomials.

For a more general $f$, we set again $P$ an interpolation polynomial such that $f^\otimes(M_1,\ldots,M_k)=P^\otimes(M_1,\ldots,M_k)$.
And so we have:
$$\begin{array}{l}f^\otimes(M_1,\ldots,M_k)^{i_1}{}_{j_1}\ldots^{i_{p-1}}{}_{j_{p-1}}{}^i{}_a{}^{i_{p+1}}{}_{j_{p+1}}\ldots^{i_{q-1}}{}_{j_{q-1}}{}^a{}_j{}^{i_{q+1}}{}_{j_{q+1}}\ldots^{i_k}{}_{j_k}=\\
P^\otimes(M_1,\ldots,M_k)^{i_1}{}_{j_1}\ldots^{i_{p-1}}{}_{j_{p-1}}{}^i{}_a{}^{i_{p+1}}{}_{j_{p+1}}\ldots^{i_{q-1}}{}_{j_{q-1}}{}^a{}_j{}^{i_{q+1}}{}_{j_{q+1}}\ldots^{i_k}{}_{j_k}=\\
P^\otimes(M_1,\ldots,M_k)^{i_1}{}_{j_1}\ldots^{i_{p-1}}{}_{j_{p-1}}{}^a{}_j{}^{i_{p+1}}{}_{j_{p+1}}\ldots^{i_{q-1}}{}_{j_{q-1}}{}^i{}_a{}^{i_{q+1}}{}_{j_{q+1}}\ldots^{i_k}{}_{j_k}=\\
f^\otimes(M_1,\ldots,M_k)^{i_1}{}_{j_1}\ldots^{i_{p-1}}{}_{j_{p-1}}{}^a{}_j{}^{i_{p+1}}{}_{j_{p+1}}\ldots^{i_{q-1}}{}_{j_{q-1}}{}^i{}_a{}^{i_{q+1}}{}_{j_{q+1}}\ldots^{i_k}{}_{j_k}=\\
Q^\otimes(M_1,\ldots,M_{q-1},M_{q+1},\ldots,M_k)^{i_1}{}_{j_1}\ldots^{i_{p-1}}{}_{j_{p-1}}{}^i{}_j{}^{i_{p+1}}{}_{j_{p+1}}\ldots^{i_{q-1}}{}_{j_{q-1}}{}^{i_{q+1}}{}_{j_{q+1}}\ldots^{i_k}{}_{j_k},\end{array}$$
where the polynomial $Q$ is defined by
$$Q(x_1,\ldots,x_{q-1},x_{q+1},\ldots,x_k)=P(x_1,\ldots,x_{q-1},x_p,x_{q+1},\ldots,x_k).$$
It remains to check that
$$\left.\left[\prod_{\substack{1\leq l\leq k\\l\neq q}}\left(\frac{\partial}{\partial x_l}\right)^{a_l}\right](g-Q)(x_1,\ldots,x_{q-1},x_{q+1},\ldots,x_k)\right|_{(x_1,\ldots,x_k)=(\lambda_{1m_1},\ldots,\lambda_{km_k})}=0,$$
with the same notation as above for the $\lambda_{lm}$'s and the same conditions for the $m_l$'s and $a_l$'s.

The left-hand side is equal to
$$\sum_{h=0}^{a_p}\binom{a_p}{h}\left[\partial_p^h\partial_q^{a_p-h}\prod_{\substack{1\leq l\leq k\\l\neq p\\l\neq q}}\partial_l^{a_l}\right](f-P)(\lambda_{1m_1},\ldots,\lambda_{q-1m_{q_1}},\lambda_{pm_p},\lambda_{q+1m_{q+1}},\ldots,\lambda_{km_k}).$$
It is $0$ because $P$ is a nice interpolation polynomial of $f$.

So we get $Q^\otimes(M_1,\ldots,M_{q-1},M_{q+1},\ldots,M_k)=g^\otimes(M_1,\ldots,M_{q-1},M_{q+1},\ldots,M_k)$ and the theorem is proved.
\end{proof}

In the case where $E_p=E_q$ but $M_p$ and $M_q$ are different, there is no such simplified expression for the contraction, but one has a kind of commutation property if $M_p$ and $M_q$ commute.

\begin{prop}\label{commut2}Let $M_1,\ldots,M_k$ be endomorphisms of finite dimensional $\C$-vector spaces $E_1,\ldots,E_k$.
Assume that for two indexes $1\leq p<q\leq k$, we have $E_p=E_q$ and the endomorphisms $M_p$ and $M_q$ commute.
Let $f$ be a function from $\C^k$ to $\C$, such that $f^\otimes(M_1,\ldots,M_k)$ is well defined.
Then we have:
$$\begin{array}{l}f^\otimes(M_1,\ldots,M_k)^{i_1}{}_{j_1}\ldots^{i_{p-1}}{}_{j_{p-1}}{}^i{}_a{}^{i_{p+1}}{}_{j_{p+1}}\ldots^{i_{q-1}}{}_{j_{q-1}}{}^a{}_j{}^{i_{q+1}}{}_{j_{q+1}}\ldots^{i_k}{}_{j_k}=\\
f^\otimes(M_1,\ldots,M_k)^{i_1}{}_{j_1}\ldots^{i_{p-1}}{}_{j_{p-1}}{}^a{}_j{}^{i_{p+1}}{}_{j_{p+1}}\ldots^{i_{q-1}}{}_{j_{q-1}}{}^i{}_a{}^{i_{q+1}}{}_{j_{q+1}}\ldots^{i_k}{}_{j_k}.\end{array}$$
\end{prop}

\begin{proof}One replaces $f$ with a suitable interpolation polynomial $P$ such that $f^\otimes(M_1,\ldots,M_k)=P^\otimes(M_1,\ldots,M_k)$.
Then using linearity, we only have to check the property for monomials, which easily follows from the well known fact that if $M_p$ and $M_q$ commute, then $M_p^{n_1}$ and $M_q^{n_2}$ also commute (the commutation of two endomorphisms $M$ and $N$ can be written $M^i{}_aN^a{}_j=M^a{}_jN^i{}_a$).
\end{proof}

\begin{rmk}We have a generalization of Remark \ref{commut1}.
If $M_1,\ldots,M_k$ and $M'_1\ldots,M'_k$ are endomorphisms of $E_1,\ldots,E_k$ such that for all $1\leq l\leq k$, $M_l$ and $M'_l$ commute, and if $f$ and $g$ are two functions from $\C^k$ to $\C$ such that $\bar{M}=f^\otimes(M_1,\ldots,M_k)$ and $\bar{M}'=g^\otimes(M'_1,\ldots,M'_k)$ are well defined, then $\bar{M}$ and $\bar{M'}$ commute as endomorphisms of $E_1\otimes\ldots\otimes E_k$.
\end{rmk}

\begin{proof}The tensors $\bar{M}\bar{M'}$ and $\bar{M'}\bar{M}$ can both be obtained by $k$ contractions from the tensor
$h^\otimes(M_1,\ldots,M_k,M'_1,\ldots,M'_k)$, where $h(x_1,\ldots,x_k,y_1,\ldots,y_k)=f(x_1,\ldots,x_k)g(y_1,\ldots,y_k)$.
The fact you get the same result follows just from applying $k$ times proposition \ref{commut2}.
\end{proof}

For a given holomorphic function $f$, the tensor $f^\otimes(M_1,\ldots,M_k)$ is an holomorphic function of $M_1,\ldots,M_k$.
The following theorem gives the expression of its derivatives.

\begin{thm}\label{diff}Let $k\in\N^*$, and $M_1,\ldots,M_k$ be endomorphisms of the finite dimensional $\C$-vector spaces $E_1,\ldots,E_k$.
For any $1\leq l\leq k$, we set $e_l$ the number of different eigenvalues of $M_l$, $\lambda_{lm}$ these eigenvalues for $1\leq m\leq e_l$, $r_{lm}$ their multiplicity as roots of the minimal polynomial of $M_l$.
Let $f$ be a function from $\C^k$ to $\C$.
Let $1\leq p\leq k$.
Assume that for all $k$-tuple $(m_1,\ldots,m_k)$, $f(x_1,\ldots,x_k)$ is holomorphic with respect to $x_p$ and all the other variables $x_l$ such that $r_{lm_l}\geq2$ near $(\lambda_{1m_1},\ldots,\lambda_{km_k})$.

Then for any endomorphism $H$ of $E_p$, we have the following:
$$\begin{array}{l}
\lim_{\varepsilon\rightarrow0}\frac{f^\otimes(M_1,\ldots,M_{p-1},M_p+\varepsilon H,M_{p+1},\ldots,M_k)^{i_1}{}_{j_1}\ldots^{i_k}{}_{j_k}-f^\otimes(M_1,\ldots,\ldots,M_k)^{i_1}{}_{j_1}\ldots^{i_k}{}_{j_k}}{\varepsilon}=\\
g^\otimes(M_1,\ldots,M_{p-1},M_p,M_p,M_{p+1},\ldots,M_k)^{i_1}{}_{j_1}\ldots^{i_{p-1}}{}_{j_{p-1}}{}^{i_p}{}_i{}^j{}_{j_p}{}^{i_{p+1}}{}_{j_{p+1}}\ldots^{i_k}{}_{j_k}H^i{}_j,\end{array}$$
with $g$ the function of $k+1$ variables defined by
$$g(x_1,\ldots,x_{p-1},x_p,y,x_{p+1},\ldots,x_{k})=\left\{\begin{array}{ll}\frac{f(x_1,\ldots,x_{p-1},y,x_{p+1},\ldots,x_k)-f(x_1,\ldots,x_k)}{y-x_p}&\textrm{if }y\neq x_p\\
\partial_pf(x_1,\ldots,x_k)&\textrm{if }y=x_p.\end{array}\right.$$
\end{thm}

\begin{rmk}\label{alterf}When $y$ is close to $x$, we can write
$$g(x_1,\ldots,x_{p-1},x_p,y,x_{p+1},\ldots,x_{k})=\int_0^1\partial_pf(x_1,\ldots,x_{p-1},(1-t)x_p+ty,x_{p+1},\ldots,x_k).$$
So we don't have regularity issues with $g$ near the hyperplane $x_p=y$.
\end{rmk}

\begin{proof}In the simple case of a monomial $f(x_1,\ldots,x_k)=\prod_{l=1}^kx_l^{\alpha_l}$, it follows from the fact that
\begin{align*}\lim_{\varepsilon\rightarrow0}\frac{((M_p+\varepsilon H)^{\alpha_p})^i{}_j-(M_p^{\alpha_p})^i{}_j}{\varepsilon}&=\sum_{h=0}^{\alpha_p-1}(M_p^{h}HM_p^{\alpha_p-1-h})^i{}_j\\
&=\left((x,y)\rightarrow\sum_{h=0}^{\alpha_p-1}x^hy^{\alpha_p-1-h}\right)^\otimes(M_p,M_p)^i{}_a{}^b{}_jH^a{}_b\end{align*}
and that we have
$$\sum_{h=0}^{\alpha_p-1}x^hy^{\alpha_p-1-h}=\left\{\begin{array}{ll}\frac{y^{\alpha_p}-x^{\alpha_p}}{y-x}&\textrm{if }x\neq y\\
\frac{\der}{\der x}x^{\alpha_p}&\textrm{if }x=y.\end{array}\right.$$

By linearity, this result extends to polynomials.

For a more general $f$, for any $1\leq l\leq k$, we set $P_l$ the minimal polynomial of $M_l$, and we set $P^{(\varepsilon)}$ the characteristic polynomial of $M_p+\varepsilon H$.
Let $Q_\varepsilon$ be the Lagrange--Sylvester interpolation polynomial of $f$ associated to the product of multisets $\prod_{l=1}^{p-1}\Root(P_l)\times\Root(P_pP^{(\varepsilon)})\times\prod_{l=p+1}^k\Root(P_l)$.
That is to say, if we set $e_l$ the number of different roots of $P_l$, $\lambda_{lm}$ these roots for $1\leq m\leq e_l$ and $r_{lm}$ their multiplicities, and if we set likewise $e^{(\varepsilon)}$ the number of different roots of $P_pP^{(\varepsilon)}$, $\lambda^{(\varepsilon)}_m$ these roots for $1\leq m\leq e^{(\varepsilon)}$ and $r^{(\varepsilon)}_m$ their multiplicities, then $Q_{\varepsilon}(x_1,\ldots,x_k)$ is the unique polynomial of $k$ variables whose degree with respect to the variable $x_l$ is at most $\deg(P_l)$, except for the variable $x_p$ for which it is at most $\deg(P_p)+\dim(E_p)$, and wich satisfies the equalities
$$\left[\prod_{l=1}^k\partial_l^{a_l}\right](Q_\varepsilon-f)(\lambda_{1m_1},\ldots,\lambda_{p-1m_{p-1}},\lambda^{(\varepsilon)}_{m_p},\lambda_{p+1m_{p+1}},\ldots,\lambda_{km_k})=0$$
for all $k$-tuples $(m_1,\ldots,m_k)$ and $(a_1,\ldots,a_k)$ such that $1\leq m_l\leq e_l$ and $0\leq a_l<r_{lm_l}$ for $l\neq p$ and $1\leq m_p\leq e^{(\varepsilon)}$ and $0\leq a_p<r^{(\varepsilon)}_{m_p}$.

Because of the continuity of the Lagrange--Sylvester interpolation, we have $Q_\varepsilon\xrightarrow[\varepsilon\rightarrow0]{}Q_0$.
Furthermore, we have
$$f^\otimes(M_1,\ldots,M_{p-1},M_p+\varepsilon H,M_{p+1},\ldots,M_k)=Q_\varepsilon^\otimes(M_1,\ldots,M_{p-1},M_p+\varepsilon H,M_{p+1},\ldots,M_k).$$
So we get
\begin{multline*}f^\otimes(M_1,\ldots,M_{p-1},M_p+\varepsilon H,M_{p+1},\ldots,M_k)-f^\otimes(M_1,\ldots,M_k)=\\
Q_\varepsilon^\otimes(M_1,\ldots,M_{p-1},M_p+\varepsilon H,M_{p+1},\ldots,M_k)-Q_0^\otimes(M_1,\ldots,M_k)\\
=(Q_\varepsilon-Q_0)^\otimes(M_1,\ldots,M_{p-1},M_p+\varepsilon H,M_{p+1},\ldots,M_k)\\
+Q_0^\otimes(M_1,\ldots,M_{p-1},M_p+\varepsilon H,M_{p+1},\ldots,M_k)-Q_0^\otimes(M_1,\ldots,M_k).
\end{multline*}
For a fixed $\varepsilon$, the difference $(Q_\varepsilon-Q_0)^\otimes(M_1,\ldots,M_{p-1},M_p+\delta H,M_{p+1},\ldots,M_k)$ is a polynomial of $\delta$ whose coefficients tend to $0$ when $\varepsilon$ tends to $0$.
The coefficient of this polynomial corresponding to $\delta^0$ is $(Q_\varepsilon-Q_0)^\otimes(M_1,\ldots,M_k)=0$.
Thus we have
$$(Q_\varepsilon-Q_0)^\otimes(M_1,\ldots,M_{p-1},M_p+\varepsilon H,M_{p+1},\ldots,M_k)=o(\varepsilon).$$
Hence, using the Theorem for the polynomial $Q_0$, we have
$$\begin{array}{l}\lim_{\varepsilon\rightarrow0}\frac{f^\otimes(M_1,\ldots,M_{p-1},M_p+\varepsilon H,M_{p+1},\ldots,M_k)^{i_1}{}_{j_1}\ldots^{i_k}{}_{j_k}-f^\otimes(M_1,\ldots,\ldots,M_k)^{i_1}{}_{j_1}\ldots^{i_k}{}_{j_k}}{\varepsilon}=\\
R^\otimes(M_1,\ldots,M_{p-1},M_p,M_p,M_{p+1},\ldots,M_k)^{i_1}{}_{j_1}\ldots^{i_{p-1}}{}_{j_{p-1}}{}^{i_p}{}_i{}^j{}_{j_p}{}^{i_{p+1}}{}_{j_{p+1}}\ldots^{i_k}{}_{j_k}H^i{}_j,\end{array}$$
with $R$ the polynomial of $k+1$ variables given by:
$$R(x_1,\ldots,x_{p-1},x_p,y,x_{p+1},\ldots,x_{k})=\left\{\begin{array}{ll}\frac{Q_0(x_1,\ldots,x_{p-1},y,x_{p+1},\ldots,x_k)-Q_0(x_1,\ldots,x_k)}{y-x_p}&\textrm{if }y\neq x_p\\
\partial_pQ_0(x_1,\ldots,x_k)&\textrm{if }y=x_p.\end{array}\right.$$

Let $(m_1,\ldots,m_p,m'_p,m_{p+1},\ldots,m_k)$ and $(a_1,\ldots,a_p,a'_p,a_{p+1},\ldots,a_k)$ be two $k+1$-tuples satisfying $1\leq m_l\leq e_l$, $1\leq m'_p\leq e_p$, $0\leq a_l<r_{lm_l}$ and $0\leq a'_p<r_{pm'_p}$.
We want to show that
$$\begin{array}{r}\left.\left[\frac{\partial}{\partial y}^{a'_p}\prod_{l=1}^k\frac{\partial}{\partial x_l}^{a_l}\right](g-R)(x_1,\ldots,x_p,y,x_{p+1},\ldots,x_k)\right|=0.\hfill\\
\scriptstyle(x_1,\ldots,x_p,y,x_{p+1},\ldots,x_k)=(\lambda_{1m_1},\ldots,\lambda_{pm_p},\lambda_{pm'_p},\lambda_{p+1m_{p+1}},\ldots,\lambda_{km_k})\end{array}$$
If $m_p\neq m'_p$, this quantity is
\begin{multline*}\scriptstyle\sum_{h=0}^{a'_p}\binom{a'_p}{h}\frac{(-1)^{a'_p-h}(a_p+a'_p-h)!\left[\partial_p^h\prod_{\substack{1\leq l\leq k\\l\neq p}}\partial_l^{a_l}\right](f-Q_0)(\lambda_{1m_1},\ldots,\lambda_{p-1m_{p-1}},\lambda_{pm'_p},\lambda_{p+1m_{p+1}},\ldots,\lambda_{km_k})}{(\lambda_{pm'_p}-\lambda_{pm_p})^{a_p+a'_p-h+1}}\\
\scriptstyle-\sum_{h=0}^{a_p}\binom{a_p}{h}\frac{(-1)^{a'_p}(a_p+a'_p-h)!\left[\partial_p^h\prod_{\substack{1\leq l\leq k\\l\neq p}}\partial_l^{a_l}\right](f-Q_0)(\lambda_{1m_1},\ldots,\lambda_{km_k})}{(\lambda_{pm'_p}-\lambda_{pm_p})^{a_p+a'_p-h+1}}\end{multline*}
which is $0$ as wanted, since each term of both sums is $0$.
If $m_p=m'_p$, using the formula of Remark \ref{alterf}, the quantity we want to compute is
$$\left[\partial_p^{a_p+a'_p+1}\prod_{\substack{1\leq l\leq k\\l\neq p}}\partial_l^{a_l}\right](f-Q_0)(\lambda_{1m_1},\ldots,\lambda_{km_k})\int_0^1(1-t)^{a'_p}t^{a_p}\der t$$
which is $0$ too because $a_p+a'_p+1<2r_{pm_p}\leq r_{pm_p}+s_{pm_p}$, where $s_{pm_p}$ is the multiplicity of $\lambda_{pm_p}$ as root of the characteristic polynomial of $M_p$.

Thus $(g-R)^\otimes(M_1,\ldots,M_{p-1},M_p,M_p,M_{p+1},\ldots,M_k)=0$ and the theorem is proved.
\end{proof}

\section{Some possible applications}
In the particular case of symmetric square matrices with real coefficients, we know that such matrices have real eigenvalues and are diagonalizable in an othonormal basis.
So we can apply real-valued functions of real variables to them, without having the regularity concerns for the definitions.
In this framework, we have the following result.

\begin{prop} We denote by $\Sym_n(\Rl)$ the vector space of symmetric $n\times n$ matrices with real coefficients, equipped with the Hilbert--Schmidt norm
$$\|M\|_{HS}=\sqrt{\Tr(M^TM)}=\sqrt{\Tr(M^2)}.$$

Let $f:\Rl\mapsto\Rl$ be a function.
Assume that $f$ is $k$-Lipschitz for some $k>0$.

Then the function
$$F:\left\{\begin{array}{rcl}\Sym_n(\Rl)&\mapsto &\Sym_n(\Rl)\\
M&\mapsto&f(M)\end{array}\right.$$
is also $k$-Lipschitz.
\end{prop}

\begin{proof}We first prove that the result holds when $f$ is a $\mathcal{C}^1$ function.
In that case, a modified version of Theorem \ref{diff} holds, so $F$ is differentiable, and its derivative at $M$ is given by
$$(\der F(M).H)^i{}_j=\lim_{\varepsilon\rightarrow 0}\frac{F(M+\varepsilon H)^i{}_j-F(M)^i{}_j}{\varepsilon}=f_1(M,M)^i{}_k{}^l{}_jH^k{}_l,$$
where we have
$$f_1(x,y)=\left\{\begin{array}{lr}\frac{f(y)-f(x)}{y-x}&\textrm{if }x\neq y\\
f'(x)&\textrm{if }x=y.\end{array}\right.$$

For a given $M\in\Sym_n(\Rl)$, we denote by $\lambda_1,\lambda_2,\ldots,\lambda_n$ its eigenvalues and $u_1,\ldots,u_n$ a set of orthonormed eigenvectors of $M$.
Then, the matrices $u_i\otimes u^\ast_i$ for $1\leq i\leq n$ and $\frac{u_i\otimes u^\ast_j+u_j\otimes u^\ast_i}{\sqrt{2}}$ for $1\leq i<j\leq n$ form an orthonormal basis of $\Sym_n(\Rl)$ and are eigenvectors of $\der F(M)$ with eigenvalues $f_1(\lambda_i,\lambda_i)$ and $f_1(\lambda_i,\lambda_j)$ respectively, according to Proposition \ref{evapp}.
Using the fact that $f$ is $k$-Lipschitz, we have $\forall x,y\in\Rl,|f_1(x,y)|\leq k$.

Thus $\der F(M)$ is a $k$-Lipschitz linear function of $\Sym_n(\Rl)$.
This being true for every $M\in\Sym_n(\Rl)$, the function $F$ itself is also $k$-Lipschitz.

If the function $f$ is not $\mathcal{C}^1$, one can approximate $f$ by a $\mathcal{C}^1$ which is also $k$-Lipschitz.
For example, we set
$$g(x)=\left\{\begin{array}{rl}0&\textrm{if }|x|\geq1\\
\frac{\nep^\frac{1}{1-x^2}}{\int_{-1}^1\nep^\frac{1}{1-y^2}\der y}&\textrm{if }|x|<1,\end{array}\right.$$
$$g_\varepsilon(x)=\frac{1}{\varepsilon}g\left(\frac{x}{\varepsilon}\right)$$
and
$$f_\varepsilon(x)=\int_\Rl f(x-y)g_\varepsilon(y)\der y=\int_\Rl f(y)g_\varepsilon(x-y)\der y.$$
Then we have
$$|f_\varepsilon(x)-f_\varepsilon(y)|=|\int_\Rl (f(x-z)-f(y-z))g_\varepsilon(z)\der z|\leq k|x-y|\int_\Rl |g_\varepsilon(z)|\der z=k|x-y|,$$
so $f_\varepsilon$ is $k$-Lipschitz.
And $f_\varepsilon$ is differentiable, its derivative being
$$f_\varepsilon'(x)=\int_\Rl f(y)g_\varepsilon'(x-y)\der y=\int_\Rl f(x-y)g_\varepsilon'(y)\der y,$$
which is continuous so $f_\varepsilon$ is $\mathcal{C}^1$.

Furthermore, we have
$$|f_\varepsilon(x)-f(x)|=|\int_\Rl (f(y)-f(x))g_\varepsilon(x-y)\der y|\leq k\int_\Rl|y|g_\varepsilon(y)\der y=k\varepsilon\int_\Rl|y|g(y)\der y.$$
So for any $M\in\Sym_n(\Rl)$, we have
\begin{multline*}\|f_\varepsilon(M)-f(M)\|_{HS}^2=\|(f_\varepsilon-f)(M)\|_{HS}^2=\Tr((f_\varepsilon-f)^2(M))=\sum_{i=1}^n(f_\varepsilon-f)^2(\lambda_i)\\
\leq k^2\varepsilon^2n\left(\int_\Rl|x|g(x)\der x\right)^2.\end{multline*}

So, for $M_1$ and $M_2$ in $\Sym_n(\Rl)$, we have
\begin{multline*}\|F(M_1)-F(M_2)\|_{HS}\leq \|F(M_1)-f_\varepsilon(M_1)\|_{HS}+\|f_\varepsilon(M_1)-f_\varepsilon(M_2)\|_{HS}+\|f_\varepsilon(M_2)-F(M_2)\|_{HS}\\
\leq k\varepsilon\sqrt{n}\int_\Rl|x|g(x)\der x+k\|M2-M_1\|_{HS}+k\varepsilon\sqrt{n}\int_\Rl|x|g(x)\der x.\end{multline*}
This inequality being true for every $\varepsilon>0$, we finally get $\|F(M_1)-F(M_2)\|_{HS}\leq k\|M2-M_1\|_{HS}$, so $F$ is k-Lipschitz.
\end{proof}

The main reason the author thinks it is a good idea to introduce $f^\otimes(M_1,\ldots, M_k)$ is that it allows to give a rather simple expression of derivatives of a function of a square matrix.
If we iterate Theorem \ref{diff}, we can get the following expression for the $n$-th derivative of $f(M)=f^\otimes(M)$.

\begin{prop}\label{iter}Let $M$ and $H$ be two endomorphisms of a finite-dimensional $\C$-vector space $E$.
Let $f:\C\mapsto\C$ be a function which is holomorphic in the neighborhood of each eigenvalue of $M$.
Then the function $F:U\subset\C\mapsto L(E,E)$ given by $F(z)=f(M+zH)$ is holomorphic in a neighborhood of $0$ and furthermore, we have:
$$F^{(n)}(z)^i{}_j=n!f_n^\otimes(\underbrace{M+zH,M+zH,\ldots,M+zH}_{n+1\textrm{ times}})^i{}_{j_1}{}^{i_1}{}_{j_2}\ldots^{i_{n-1}}{}_{j_n}{}^{i_n}{}_jH^{j_1}{}_{i_1}\ldots H^{j_n}{}_{i_n},$$
where $U$ is an open subset of $\C$ which contains $0$, and
$$f_n(x_0,\ldots,x_n)=f[x_0,\ldots,x_n]$$
is the (generalized) divided difference of the function $f$ on the nodes $x_0,\ldots,x_n$.
\end{prop}
\begin{proof}We proceed by induction on $n$.
The case $n=0$ is trivial.

Assume that
$$F^{(n)}(z)^i{}_j=n!f_n^\otimes(M+zH,\ldots,M+zH)^i{}_{j_1}{}^{i_1}{}_{j_2}\ldots^{i_{n-1}}{}_{j_n}{}^{i_n}{}_jH^{j_1}{}_{i_1}\ldots H^{j_n}{}_{i_n}.$$
then if we want to differentiate this expression one more time with respect to $z$, we have to differentiate $f_n^\otimes(M+zH,\ldots,M+zH)$ with respect to each (matricial) variable.
So using Theorem \ref{diff}, we get
\begin{align*}F^{(n+1)}(z)^i{}_j&=n!\sum_{k=0}^nf_{n+1}^\otimes(M+zH,\ldots,M+zH)^i{}_{j_1}{}^{i_1}{}_{j_2}\ldots^{i_k}{}_{i'}{}^{j'}{}_{j_{k+1}}\ldots^{i_{n-1}}{}_{j_n}{}^{i_n}{}_jH^{i'}{}_{j'}H^{j_1}{}_{i_1}\ldots H^{j_n}{}_{i_n}\\
&=n!\sum_{k=0}^nf_{n+1}^\otimes(M+zH,\ldots,M+zH)^i{}_{j_1}{}^{i_1}{}_{j_2}\ldots^{i_n}{}_{j_{n+1}}{}^{i_{n+1}}{}_jH^{j_1}{}_{i_1}\ldots H^{j_{n+1}}{}_{i_{n+1}}\\
&=(n+1)!f_{n+1}^\otimes(M+zH,\ldots,M+zH)^i{}_{j_1}{}^{i_1}{}_{j_2}\ldots^{i_n}{}_{j_{n+1}}{}^{i_{n+1}}{}_jH^{j_1}{}_{i_1}\ldots H^{j_{n+1}}{}_{i_{n+1}},
\end{align*}
where we just relabelled the indexes to derive line 2 from line 1.
So the induction hypothesis is true at the rank $n+1$.
\end{proof}

\begin{rmk}In the case where $H$ and $M$ commute, we can get a simpler expression.
Indeed, we can write
$$F^{(n)}(z)^i{}_j=n!g_n^\otimes(M+zH,\ldots,M+zH,H,\ldots,H)^i{}_{j_1}{}^{i_1}{}_{j_2}\ldots^{i_{n-1}}{}_{j_n}{}^{i_n}{}_j{}^{j_1}{}_{i_1}\ldots ^{j_n}{}_{i_n},$$
where $g_n$ is the function of $2n+1$ variables defined by
$$g_n(x_0,\ldots,x_n,y_1,\ldots,y_n)=f[x_0,\ldots,x_n]y_1\ldots y_n.$$
Since $H$ and $M$ commute, $H$ and $M+zH$ commute too.
Using Proposition \ref{commut2} sufficiently many times, we can get
$$F^{(n)}(z)^i{}_j=n!g_n^\otimes(M+zH,\ldots,M+zH,H,\ldots,H)^i{}_{j_1}{}^{j_1}{}_{i_1}{}^{i_1}{}_{j_2}\ldots^{j_n}{}_{i_n}{}^{i_n}{}_j.$$
Using Theorem \ref{contr2}, one gets:
$$F^{(n)}(z)^i{}_j=n!\bar{g}_n^\otimes(M+zH,H)^i{}_k{}^k{}_j$$
where $\bar{g}_n(x,y)=g_n(x,\ldots,x,y,\ldots,y)=f[x,\ldots,x]y^n=\frac{f^{(n)}(x)}{n!}y^n$.
So finally, we get:
$$F^{(n)}(z)=f^{(n)}(M+zH)H^n.$$
\end{rmk}

\begin{rmk}\label{trace}If we differentiate $\Tr(f(M+zH))$, one gets
$$F^{(n)}(z)^i{}_i=n!f_n^\otimes(M+zH,\ldots,M+zH)^i{}_{j_1}{}^{i_1}{}_{j_2}\ldots^{i_{n-1}}{}_{j_n}{}^{i_n}{}_iH^{j_1}{}_{i_1}\ldots H^{j_n}{}_{i_n}.$$
Using Theorem $\ref{contr2}$, we get
$$F^{(n)}(z)^i{}_i=n!h_{n,n}^\otimes(M+zH,\ldots,M+zH)^{i_n}{}_{j_1}{}^{i_1}{}_{j_2}\ldots^{i_{n-1}}{}_{j_n}H^{j_1}{}_{i_1}\ldots H^{j_n}{}_{i_n},$$
with $h_{n,n}$ the function of $n$ variables defined by
$$h_{n,n}(x_1,\ldots,x_n)=f_n(x_n,x_1,\ldots,x_n).$$
With $n=1$, one gets the classical result that
$$\frac{\der}{\der z}\Tr(f(M+zH))=\Tr(f'(M+zH)H).$$
If one differentiates $n-1$ extra times this formula, one gets
$$F^{(n)}(z)^i{}_i=(n-1)!f'^\otimes_{n-1}(M+zH,\ldots,M+zH)^{i_n}{}_{j_1}{}^{i_1}{}_{j_2}\ldots^{i_{n-1}}{}_{j_n}H^{j_1}{}_{i_1}\ldots H^{j_n}{}_{i_n},$$
which seems to be different from the other formula above.
In fact we get the same thing because if we take $h_{n,k}(x_1,\ldots,x_n)=h_{n,n}(x_{k+1},\ldots,x_n,x_1,\ldots,x_k)$, we have
\begin{multline*}h_{n,n}^\otimes(M+zH,\ldots,M+zH)^{i_n}{}_{j_1}{}^{i_1}{}_{j_2}\ldots^{i_{n-1}}{}_{j_n}H^{j_1}{}_{i_1}\ldots H^{j_n}{}_{i_n}=\\
h_{n,k}^\otimes(M+zH,\ldots,M+zH)^{i_n}{}_{j_1}{}^{i_1}{}_{j_2}\ldots^{i_{n-1}}{}_{j_n}H^{j_1}{}_{i_1}\ldots H^{j_n}{}_{i_n}\end{multline*}
by relabelling the indexes, and the equality
$$f'_{n-1}(x_1,\ldots,x_n)=\sum_{k=1}^nh_{n,k}(x_1,\ldots,x_n)$$
holds.
\end{rmk}

\begin{rmk}Despite the fact that the tensor $f_n^\otimes(M,\ldots,M)$ has a lot of symmetries (due to the fact that $f_n$ is a symmetric function), the $n$-linear application $F^{(n)}(0).(H_1,\ldots,H_n)$ obtained by polarization is not in general given by $n!$ times this tensor, but by a symmetrization of it.
\begin{align*}F^{(n)}(0).(H_1,\ldots,H_n)^i{}_j&=\left(\sum_{\sigma\in\Perm_n}f_n^\otimes(M,\ldots,M)^i{}_{j_{\sigma(1)}}{}^{i_{\sigma(1)}}{}_{j_{\sigma(2)}}\ldots^{i_{\sigma(n-1)}}{}_{j_{\sigma(n)}}{}^{i_{\sigma(n)}}{}_j\right)\\
&\qquad\qquad(H_1)^{j_1}{}_{i_1}\ldots(H_n)^{j_n}{}_{i_n}\\
&\neq n!f_n^\otimes(M,\ldots,M)^i{}_{j_1}{}^{i_1}{}_{j_2}\ldots^{i_{n-1}}{}_{j_n}{}^{i_n}{}_j(H_1)^{j_1}{}_{i_1}\ldots(H_n)^{j_n}{}_{i_n},\end{align*}
where $\Perm_n$ is the set of permutations of $\{1,\ldots,n\}$.

For example, if $f(z)=z^n$, we have $f_n=1$ and then $f_n^\otimes(M,\ldots,M)=I^{\otimes n+1}$, but we have
$$\sum_{\sigma\in \Perm_n}H_{\sigma(1)}\ldots H_{\sigma(n)}\neq n!H_1\ldots H_n.$$
\end{rmk}

If a function $f:\C\mapsto\C$ is $\Rl$-differentiable but not holomorphic, then $f(M)$ is well defined if $M$ has no multiple eigenvalues.
One can wonder if $f(M)$ is $\Rl$-differentiable.
The answer is yes, but to the author's knowledge, one does not have a nice expression of the derivative like we have in the holomorphic case.
But we can use the fact that
$$f(M)=\sum_{\lambda\textrm{ eigenvalue of }M}f(\lambda)P_M(\lambda),$$
where $P_M(\lambda)$ is the projection on the eigensubpace of $M$ corresponding to the eigenvalue $\lambda$, parallelwise to all the other eigensubspaces of $M$.

We set $M(t)=M+tH$.
Because of the continuity of eigenvalues, there exist continuous functions $\lambda_k(t)$ for $1\leq k\leq \dim(E)$, defined in a neighborhood of $0$, such that for a given $t$, the $\lambda_k(t)$ are the eigenvalues of $M(t)$.
The projectors $P_{M(t)}(\lambda_k(t))$ are also continuous.
As we will see below, the $\lambda_k(t)$ and $P_{M(t)}(\lambda_k(t))$ are analytic.
So if we set $F(t)=f(M(t))$, we have $$F(t)=\sum_{k=1}^{\dim(E)}f(\lambda_k(t))P_{M(t)}(\lambda_k(t))$$
Thus if $f$ is $n$ times $\Rl$-differentiable, we get
$$F^{(n)}(t)=\sum_{k=1}^{\dim(E)}\sum_{h=0}^n\binom{n}{h}\frac{\der}{\der t}^hf(\lambda_k(t))\frac{\der}{\der t}^{n-h}P_{M(t)}(\lambda_k(t)).$$
The derivatives of $f(\lambda_k(t))$ can be expressed with the derivatives of $f$ and $\lambda_k$ thanks to the Faa di Bruno formula.

Now we look at the behaviour of $\lambda_k(t)$ and $P_{M(t)}(\lambda_k(t))$ near $0$.
Since the eigenvalues of $M$ are different, we can set $2\delta>0$ the minimum of the distance between two of them.
So the balls $\ball(\lambda_k,\delta)$ do not overlap (with $\lambda_k=\lambda_k(0)$).
Because of the continuity of eigenvalues, there exists $\varepsilon>0$ such that for $|t|\leq\varepsilon$, we have $\lambda_k(t)\in\ball(\lambda_k,\delta)$ for all $1\leq k\leq \dim(E)$.
Let us denote by $u^{(\lambda)}_{(\delta)}$ the characteristic function of $\ball(\lambda,\delta)$, which is holomorphic everywhere except on the boundary of $\ball(\lambda,\varepsilon)$.
Then we have, for $t<\varepsilon$,
\begin{align*}\lambda_k(t)&=\Tr(M(t)u^{(\lambda_k)}_{(\delta)}(M(t)))\\
P_{M(t)}(\lambda_k(t))&=u^{(\lambda_k)}_{(\delta)}(M(t)).
\end{align*}

To get the derivatives of these two guys, one can use Proposition \ref{iter}, but using the fact that the derivative of $zu^{(\lambda_k)}_{(\delta)}(z)$ is $u^{(\lambda_k)}_{(\delta)}(z)$ and Remark \ref{trace}, we get $\lambda_k^{(n)}(t)=\Tr\left(\frac{\der}{\der t}^{n-1}\left(P_{M(t)}(\lambda_k(t))\right)H\right)$, so we only have to look at the derivatives of $P_{M(t)}(\lambda_k(t))$.
Using Proposition \ref{iter}, we get:
$$\frac{\der}{\der t}^{n}(P_{M(t)}(\lambda_k(t)))^i{}_j=n!u^{(\lambda_k)\otimes}_{(\delta)n}(M(t),\ldots,M(t))^i{}_{j_1}{}^{i_1}{}_{j_2}\ldots^{i_{n-1}}{}_{j_n}{}^{i_n}{}_jH^{j_1}{}_{i_1}\ldots H^{j_n}{}_{i_n}.$$
If we compute that at $t=0$, we can replace $u^{(\lambda_k)}_{(\delta)}$ with $u^{(\lambda_k)}_{(\delta')}$ with $\delta'\leq\delta$, since this two functions coincide on a neighborhood of the spectrum of $M$.

We have the following interesting fact: for all $n$ the functions $u^{(\lambda)}_{(\delta)n}$ simply converge to a limit we will denote by $u^{(\lambda)}_n$ when $\delta$ tends to $0$ (these functions are not defined on the whole space $\C^{n+1}$, but for any single point, there are only finitely many bad $\delta$'s, and furthermore, for any $(x_0,\ldots,x_n)$, the function $\delta\mapsto u^{(\lambda)}_{(\delta)n}(x_0,\ldots,x_n)$ is constant by parts).

The function $u^{(\lambda)}_n$ is the symmetric function of $n+1$ variables such that for all $0\leq m\leq n+1$ and $z_0,\ldots,z_{n-m}\neq\lambda$, we have
$$u^{(\lambda)}_n(\underbrace{\lambda,\ldots,\lambda}_{m\textrm{ times}},z_0,\ldots,z-m)=\left\{\begin{array}{ll}0&\textrm{if }m=0\\
\begin{array}{c}\frac{1}{(m-1)!}\left.\frac{\der}{\der z}^{m-1}\prod_{h=0}^{n-m}\frac{1}{z-z_h}\right|_{z=\lambda}\\
\scriptstyle=(-1)^{(m-1)}\sum_{k_0+\ldots+k_{n-m}=m-1}\prod_{h=0}^{n-m}\frac{1}{(\lambda-z_h)^{1+k_h}}\end{array}&\textrm{else.}\end{array}\right.$$

So we get
$$\frac{\der}{\der t}^{n}(P_{M(t)}(\lambda_k(t)))^i{}_j=n!u^{(\lambda_k(t))\otimes}_n(M(t),\ldots,M(t))^i{}_{j_1}{}^{i_1}{}_{j_2}\ldots^{i_{n-1}}{}_{j_n}{}^{i_n}{}_jH^{j_1}{}_{i_1}\ldots H^{j_n}{}_{i_n}.$$
And finally, we have
$$\begin{array}{l}\scriptstyle\frac{\der}{\der t}^{n}(P_{M(t)}(\lambda_k(t)))=n!\sum_{1\leq k_0,\ldots,k_n\leq\dim(E)}u^{(\lambda_k(t))}_n(\lambda_{k_0}(t),\ldots,\lambda_{k_n}(t))P_{M(t)}(\lambda_{k_0}(t))HP_{M(t)}(\lambda_{k_1}(t))H\ldots HP_{M(t)}(\lambda_{k_n}(t))\\
\scriptstyle\lambda^{(n)}_k(t)=(n-1)!\sum_{1\leq k_0,\ldots,k_{n-1}\leq\dim(E)}u^{(\lambda_k(t))}_{n-1}(\lambda_{k_0}(t),\ldots,\lambda_{k_{n-1}}(t))\Tr(P_{M(t)}(\lambda_{k_0}(t))H\ldots HP_{M(t)}(\lambda_{k_{n-1}}(t))H).
\end{array}$$

The following proposition shows how we can get $\sum_{i_1\neq i_2\neq\ldots\neq i_k}f(\lambda_{i_1},\ldots,\lambda_{i_k})$ from the tensor $f^\otimes(M,\ldots,M)$, where the $\lambda_i$'s are the eigenvalues of $M$.
This is a generalization of $\Tr(f(M))=\sum_if(\lambda_i)$.

\begin{prop}Let $M$ be an endomorphism of a $\C$-vector space $E$ of finite dimension $d$.
Let $f$ be a function of $k$ variables such that $f^\otimes(M,\ldots,M)$ is well defined.
We set $\lambda_1,\ldots,\lambda_d$ the eigenvalues of $M$ (appearing with their multiplicities).
Then we have the following:
$$\sum_{\substack{1\leq n_1,\ldots,n_k\leq d\\\forall l\neq l',n_l\neq n_{l'}}}f(\lambda_{n_1},\ldots,\lambda_{n_k})=k!f^\otimes(M,\ldots,M)^{i_1}{}_{j_1}\ldots^{i_k}{}_{j_k}(\Pi^\wedge_k)^{j_1}{}_{i_1}\ldots^{j_k}{}_{i_k},$$
where
$$(\Pi^\wedge_k)^{j_1}{}_{i_1}\ldots^{j_k}{}_{i_k}=\frac{1}{k!}\sum_{\sigma\in\Perm_k}\epsilon(\sigma)I^{j_1}{}_{i_{\sigma(1)}}I^{j_2}{}_{i_{\sigma(2)}}\ldots I^{j_k}{}_{i_{\sigma(k)}},$$
with $\epsilon(\sigma)$ the signature of the permutation $\sigma$.
The tensor $(\Pi^\wedge_k)$ is the one corresponding to the canonical projection from $E^{\otimes k}$ to the subspace $E^{\wedge k}$ of antisymmetric tensors of order $k$.
\end{prop}

\begin{proof}Let $\sigma\in\Perm_k$.
We set $\Omega(\sigma)$ the number of orbits of $\sigma$, and we denote by $\omega_0,\ldots,\omega_{\Omega(\sigma)}$ these orbits.
For $1\leq l\leq k$, we set $1\leq c(l)\leq \Omega(\sigma)$ the unique index such that $l\in\omega_{c(l)}$.

Then, using theorems \ref{contr1} and \ref{contr2}, we have the following:
$$f^\otimes(M,\ldots,M)^{i_1}{}_{j_1}\ldots^{i_k}{}_{j_k}I^{j_1}{}_{i_{\sigma(1)}}I^{j_2}{}_{i_{\sigma(2)}}\ldots I^{j_k}{}_{i_{\sigma(k)}}=\sum_{1\leq n_1,\ldots,n_{\Omega(\sigma)}\leq d}f(\lambda_{n_{c(1)}},\ldots,\lambda_{n_{c(k)}}).$$

Now, for $1\leq n_1,\ldots,n_k\leq d$, we look at the number of times the term $f(\lambda_{n_1},\ldots,\lambda_{n_k})$ appears when we use the formula above to compute
$$\sum_{\sigma\in\Perm_k}\epsilon(\sigma)f^\otimes(M,\ldots,M)^{i_1}{}_{j_1}\ldots^{i_k}{}_{j_k}I^{j_1}{}_{i_{\sigma(1)}}I^{j_2}{}_{i_{\sigma(2)}}\ldots I^{j_k}{}_{i_{\sigma(k)}}.$$
This term appears once for each $\sigma$ such that $l\mapsto n_l$ is constant on all the orbits of $\sigma$, with the prefactor $\epsilon(\sigma)$.
The set of such permutations is in fact a subgroup of $\Perm_k$, more precisely the one of permutations which stabilize the equivalence classes of the relation $l\sim l'\Leftrightarrow n_l=n_{l'}$.

If there exist two different indexes $p$ and $q$ such that $n_p$ and $n_q$ are equal, then the above-mentionned subgroup contains the transposition $(p,q)$ which has signature $-1$, and thus, since $\epsilon$ is a group morphism, half of the elements of the subgroup has signature $1$ and the other half has signature $-1$, so the sum of the signatures is $0$.

If all the $n_l$ are different, then the subgroup is just the identity, so our term only appears once.

Hence we have
$$\sum_{\sigma\in\Perm_k}\epsilon(\sigma)f^\otimes(M,\ldots,M)^{i_1}{}_{j_1}\ldots^{i_k}{}_{j_k}I^{j_1}{}_{i_{\sigma(1)}}I^{j_2}{}_{i_{\sigma(2)}}\ldots I^{j_k}{}_{i_{\sigma(k)}}=\sum_{\substack{1\leq n_1,\ldots,n_k\leq d\\\forall l\neq l',n_l\neq n_{l'}}}f(\lambda_{n_1},\ldots,\lambda_{n_k}),$$
as wanted.
\end{proof}

A classical example is given by taking $k=d$, and $f(x_1,\ldots,x_d)=x_1\ldots x_d$, in which case we can get
$$\det(M)=\sum_{a_1+2a_2+\ldots+da_d=d}\frac{(-1)^{d-a_1-a_2-\ldots-a_d}}{a_1!a_2!\ldots a_d!1^{a_1}2^{a_2}\ldots d^{a_d}}\Tr(M)^{a_1}\Tr(M^2)^{a_2}\ldots\Tr(M^d)^{a_d}.$$

The restriction to $E^{\wedge k}$ of $\Pi^\wedge_kf^\otimes(M,\ldots,M)$, where $\Pi^\wedge_k$ and $f^\otimes(M,\ldots,M)$ are seen as endomorphisms of $E^\otimes k$, can also be itself interesting (and not only its trace).
Let us denote it by $f^\wedge(M,\ldots,M)$.

Indeed, since $E^{\wedge k}$ has dimension $\binom{d}{k}$, $f^\wedge(M,\ldots,M)$ lives in a space of dimension $\binom{d}{k}^2$, whereas $f^\otimes(M,\ldots,M)$ lives in a space of much greater dimension $d^{2k}$.

The eigenvalues of $f^\wedge(M,\ldots,M)$ are $\frac{1}{k!}\sum_{\sigma\in\Perm_k}f(\lambda_{n_{\sigma(1)}},\ldots,\lambda_{n_{\sigma(k)}})$ for $1\leq n_1<n_2<\ldots<n_k\leq d$ (the corresponding eigenvector is $v_{n_1}\wedge\ldots\wedge v_{n_k}$ if $M$ is diagonalizable and if $v_1,\ldots,v_d$ is a basis of eigenvectors of $M$, such that $v_n$ is an eigenvector of $M$ for the eigenvalue $\lambda_n$).

One could want to extend the definition of $f^\wedge(M,\ldots,M)$ to some cases in which $f^\otimes(M,\ldots,M)$ is not well defined.
For example, for $f(x,y)=\frac{1}{(x-y)^2}$, $f^\wedge(M,M)$ should have a sense if $M$ only has simple eigenvalues, whereas $f^\otimes(M,M)$ has no sense, whatever $M$ is.
But the minimal conditions we should put on $f$ and the $M_l$'s to extend the definition of $f^\wedge(M_1,\ldots,M_k)$ are not clear.

If the fact that $f^\otimes(M_1,\ldots,M_k)$ is a tensor and not a matrix is disturbing, there could be a way to define a kind of $f(M_1,\ldots,M_k)$ which would be a matrix.
One way to get a matrix from the tensor $f^\otimes(M_1,\ldots,M_k)$ is to use contractions.
Assume that one can write $f(x_1,\ldots,x_k)=\sum a_{n_1,\ldots,n_k}x_1^{n_1}\ldots x_k^{n_k}$, where we have $\sum |a_{n_1,\ldots,n_k}|\rho_1^{n_1}\ldots\rho_k^{n_k}<\infty$ for some $\rho_1,\ldots,\rho_k>0$.
Assume that for every $1\leq l\leq k$, we have $\sup_{\lambda\textrm{ eigenvalue of }M_l}|\lambda|<\rho_l$.
Then the sum:
$$\sum a_{n_1,\ldots,n_k}(M_1^{n_1}M_2^{n_2}\ldots M_k^{n_k})^i{}_j$$
is convergent, and this matrix is exactly
$$f^\otimes(M_1,\ldots,M_k)^i{}_{i_1}{}^{i_1}{}_{i_2}\ldots^{i_{k-2}}{}_{i_{k-1}}{}^{i_{k-1}}{}_j.$$
It is what we expect $f(M_1,\ldots,M_k)$ to be when all the $M_l$ commute.
But if they do not commute, this may introduce a dissymetry.
It does still work well if $f(x_1,\ldots,x_k)=\sum_{l=1}^kf_l(x_l)$.

But in the simple example $f(x,y)=(x+y)^2$, we have $f^\otimes(A,B)^i{}_k{}^k{}_j=(A^2+B^2+2AB)^i{}_j$ and not $((A+B)^2)^i{}_j$ as one could want.

If $f(x,y)=\frac{1}{x+y}$, and provided none of the eigenvalues of $A$ is the opposite of an eigenvalue of $B$, it gives you the unique matrix $M$ such that $AM+MB=I$, and this matrix is not $(A+B)^{-1}$ in general.
\bibliographystyle{abbrv}
\bibliography{base}

\begin{thebibliography}{1}

\bibitem{GNT59}
F.~R. Gantmacher.
\newblock Matrix theory.
\newblock {\em Chelsea, New York}, 21, 1959.

\end{thebibliography}
\end{document}